\documentclass[11pt,a4paper]{amsart}
\usepackage{upgreek,subfig}
\usepackage{amssymb,amsfonts,amsmath,caption,floatrow}
\usepackage{xcolor,graphicx}
\usepackage[T1]{fontenc}
\usepackage{hyperref}
\usepackage[belowskip=-10pt,aboveskip=2pt]{caption}
 \newtheorem{thm}{Theorem}[section]
 \newtheorem{cor}[thm]{Corollary}
 \newtheorem{lem}[thm]{Lemma}
 
 \theoremstyle{definition}
 \newtheorem{defn}[thm]{Definition}
 \theoremstyle{remark}
 
 \numberwithin{equation}{section}
\newcommand{\myp}[2][0cm]{\mathopen{}\left(#2\parbox[h][#1]{0cm}{}\right)}
\begin{document}
\title[Solution of the logarithmic coefficients conjecture $\ldots$]
 {Solution of the logarithmic coefficients conjecture in some families of univalent functions}

\author[S. Kanas and V. S. Masih]{ Stanislawa Kanas$^{1}$ and Vali Soltani Masih}
\address{$^{1}$University of Rzeszow, Al. Rejtana 16c, PL-35-959 Rzesz\'{o}w, Poland }
\email{skanas@ur.edu.pl}
\address{Department of Mathematics\\ Payame Noor University\\ Tehran, Iran}
\email{masihvali@gmail.com; v\_soltani@pnu.ac.ir}
\thanks{$^{1}$Corresponding author}

\begin{abstract}
For univalent and normalized functions $f$ the logarithmic coefficients $\gamma_n(f)$ are determined by the formula $\log(f(z)/z)=\sum_{n=1}^{\infty}2\gamma_n(f)z^n$. 
In the paper \cite{Pon} the authors posed the conjecture that a locally univalent function in the unit disk, satisfying  the condition
\[
\Re\left\{1+zf''(z)/f'(z)\right\}<1+\lambda/2\quad (z\in \mathbb{D}),
\]
 fulfill also the following inequality:
$$|\gamma_n(f)|\le \lambda/(2n(n+1)).$$
Here $\lambda$ is a real number such that $0<\lambda\le 1$.
In the paper we confirm that the conjecture is true, and sharp. 
\end{abstract}

\maketitle
\section{Introduction}\label{intro}

Let $\mathcal{A}$ denote the class of \emph{holomorphic functions}
in the open unit disc $\mathbb{D}=\left\{z\colon  |z|<1\right\}$ on the complex plane $\mathbb{C}$ of the form
\begin{equation}\label{eq_expand_f}
    f(z)=z+\sum_{n=2}^{\infty}a_nz^n \quad \left(z \in \mathbb{D}\right).
\end{equation}
The subclass of $\mathcal{A}$  consisting of all \emph{univalent}
functions $f$ in $\mathbb{D}$, is denoted by $\mathcal{S}$. By $\mathcal{CV}$ and $\mathcal{ST}$ we denote the subfamilies of $f\in\mathcal{S}$ of
\textit{convex and starlike} functions (resp.).  Functions $f\in \mathcal{CV}$ ($f\in \mathcal{ST}$) are analytically characterized by the condition $\Re\{1+zf''(z)/f'(z)\} > 0$ ($\Re\{zf'(z)/f(z)\} > 0$, resp.) in $\mathbb{D}$.
A natural generalization of the mentioned classes are $\mathcal{CV}(\beta)$  and $\mathcal{ST}(\beta)$, of \textit{convex functions of order} $\beta$ (\textit{starlike functions of order}$\beta$ resp.), where $0\le \beta<1$, consisting of functions $f$ that satisfy  $\Re\{1+zf''(z)/f'(z)\} > \beta$ ($\Re\{zf'(z)/f(z)\} > \beta$, resp.) in $\mathbb{D}$.

For  $f\in \mathcal{S}$ a function
\begin{equation}\label{eq_expand_log}
    \log\dfrac{f(z)}{z}=\sum_{n=1}^\infty 2\gamma_n(f)z^n \quad \myp{z\in
    \mathbb{D}},
\end{equation} is well defined and  $\gamma_n(f)$ are called \textit{logarithmic coefficients} of  $f$ \cite{Dur}. The importance of the logarithmic coefficients follows from the long-known Lebedev–Milin inequalities, where estimates of the logarithmic coefficients were used to find bounds  of the coefficients of $f$. Differentiation of \eqref{eq_expand_log} and equation coefficients of both sides implies
\begin{equation}
{\gamma_{1}(f)=\frac{1}{2} a_{2}}, \quad {\gamma_{2}(f)=\frac{1}{2}\left(a_{3}-\frac{1}{2} a_{2}^{2}\right)}.\end{equation}
The classical coefficients bound and  use of the Fekete-Szeg\"{o} inequality in $\mathcal{S}$ yields the sharp estimates
\[
|\gamma_1(f)|\le 1,\qquad \left|\gamma_{2}(f)\right| \leq \frac{1}{2}\left(1+2 e^{-2}\right).
\]
Despite this the problem of the best upper bounds for $\left|\gamma_n(f)\right|\ (n=3,4,...)$ in $ \mathcal{S}$ is still open. To be closer to the accurate solution  the estimates of $\left|\gamma_n(f)\right|$ are found in many different subclasses of $\mathcal{S}$. For example in the class of univalent and starlike functions  the sharp bound  $|\gamma_n(f)|\le 1/n$ holds true, however it is not true in  whole class $\mathcal{S}$. Although, the Koebe function $k(z)=z\left(1-\mathrm{e}^{\mathrm{i}\theta}z\right)^{-2}$ for each real $\theta$, has logarithmic coefficients $\gamma _n(k)={e^{\mathrm{i}n\theta}}/{n}\ (n\geq 1)$ then there exists a bounded, univalent function $f$  with $\gamma_n(f)={\rm O}\myp{n^{-0.83}}$ (see, for example \cite[Theorem 8.4]{Dur}).
Logarithmic coefficients problem was also considered in the classes $\mathcal{G}(\lambda)$ and $\mathcal{N}(\lambda)$, defined below, and a partial results were obtained \cite{Pon}.
We remind that a locally univalent function $f$ is a member of $\mathcal{G}(\lambda)\ (0<\lambda\le 1)$, if it satisfies the condition
\begin{equation}\label{Glambda}
\Re\left(1+zf''(z)/f'(z)\right)<1+\lambda/2\quad \myp{z\in \mathbb{D}},
 \end{equation}
 and $f\in \mathcal{A}$ is an element of  $\mathcal{N}(\lambda)\ (0<\lambda\leq1)$, if 
\begin{equation}\label{Nlambda}
 \Re\left(zf'(z)/f(z)\right)<1+\lambda/2\quad \left( z\in \mathbb{D}\right).
 \end{equation}
Between $\mathcal{G}(\lambda)$ and $\mathcal{N}(\lambda)$ it holds the standard Alexander relation: $f\in \mathcal{G}(\lambda)$ if and only if $zf'\in \mathcal{N}(\lambda)$.

The class  $\mathcal{G}(\lambda)$ was studied by several mathematicians; for example Masih et al.  proved that functions in $\mathcal{G}(\lambda)$ are close-to-convex and univalent in $\mathbb{D}$ \cite{MES} (see also \cite{Ozaki}).  Umezawa discussed a general version of the conditions \eqref{Glambda} and \eqref{Nlambda} \cite{umezawa}. The functions from the family $\mathcal{G}(\lambda)$ are starlike, see e.g. \cite{PR, SS} and also \cite{JO}. Recently, Ponnusamy et al. \cite{PSW} proved the sharp estimates for the initial logarithmic coefficients of $f\in \mathcal{G}(\lambda)$
\[
\left|\gamma_{1}(f)\right| \leq \frac{\lambda}{4}, \quad\left|\gamma_{2}(f)\right| \leq \frac{\lambda}{12}, \quad \left|\gamma_{3}(f)\right| \leq \frac{\lambda}{24},
\]
with the equality attained for the function $f$ such that $f'(z)=\myp{1-z^n}^{\lambda/n}\ (n=1,2,3)$. Also, they posed the conjecture that for $f\in \mathcal{G}(\lambda)\ (0<\lambda\le 1)$, the following sharp inequalities
\[
\left|\gamma_n(f)\right|\le \frac{\lambda}{2n(n+1)}\quad (n=1,2,...)
\]
holds. In the present paper we aim to study the  logarithmic coefficients $|\gamma_n(f)|$ for $f\in\mathcal{G}(\lambda)$ and prove that the above conjecture is true and sharp.  To do this, we  establish  a  correspondence between family $\mathcal{G}(\lambda)$  and a family related to sinusoidal spiral, denoted  $\mathcal{ST}_{ss}(\lambda)$. Additionally, we prove the growth, distortion and rotation theorem for functions $f$ in the class $\mathcal{G}(\lambda)$ and $\mathcal{N}(\lambda)$ and present some coefficient estimates in $\mathcal{ST}_{ss}(\lambda)$.

We also note that in the various families of analytic functions the absolute square series of logarithmic coefficients satisfy the sharp inequalities. For example, the logarithmic coefficients $\gamma_n(f)$ of every function $f\in \mathcal{ST}(\beta)$ satisfy the sharp inequality
\[
\sum_{n=1}^{\infty}\left|\gamma_{n}(f)\right|^{2} \leq (1-\beta)^{2} \frac{\pi^{2}}{6}\quad \myp{0\le \beta<1},
\]
with the equality for the Koebe function of order $\beta$ \cite{OPW1}. Also, in the class $\mathcal{G}(\lambda)$ the logarithmic coefficients  satisfy the inequalities \cite{PSW}
\begin{eqnarray*}
\sum_{n=1}^{\infty} n^{2}\left|\gamma_{n}(f)\right|^{2} &\le& \frac{\lambda}{4(\lambda+2)}, \\
 \sum_{n=1}^{\infty}\left|\gamma_{n}(f)\right|^{2} &\le& \frac{\lambda^{2}}{4} {\rm Li}_{2}\left((1+\lambda)^{-2}\right),
\end{eqnarray*}
where $\mathrm{Li}_2$ (dilogarithm function) is defined as follows
\[
\mathrm{Li}_2(z)=\sum_{n=1}^{\infty}\frac{z^n}{n^2}=-\int_{0}^{z}\frac{\ln (1-t)}{t}\,\mathrm{d}t \qquad \myp{z\in \mathbb{D}}.
\]

Very recently, a new subfamily related to the domain bounded by the sinusoidal spiral
\begin{eqnarray*}
        \mathbb{SS}(\lambda)&=&\left\{\rho\mathrm{e}^{\mathrm{i} \varphi}\colon\quad \rho
        =\myp{2\cos\frac{\varphi}{\lambda}}^{\lambda},\quad-\frac{\lambda \pi}{2}<\varphi\le\frac{\lambda \pi}{2}\right\}\\
        &=&\left\{w\in \mathbb{C}\colon\quad \Re\left\{w\right\}>0,\quad \Re\left\{w^{-1/\lambda}\right\}=\frac12\right\}\cup\{0\},
\end{eqnarray*}
 with $0<\lambda\le 1$, was defined \cite{MES}. The sinusoidal spiral   $\mathbb{SS}(\lambda)$ intersects the real axis at the origin and $(u,0)=(2^{\lambda},0)$ and its maximal slope angle to the real axis equals $(\pi \lambda)/2$, see Fig.\ref{Fig}.
  \begin{figure}[!h]
\centering
{%
\includegraphics[width=0.5\textwidth]{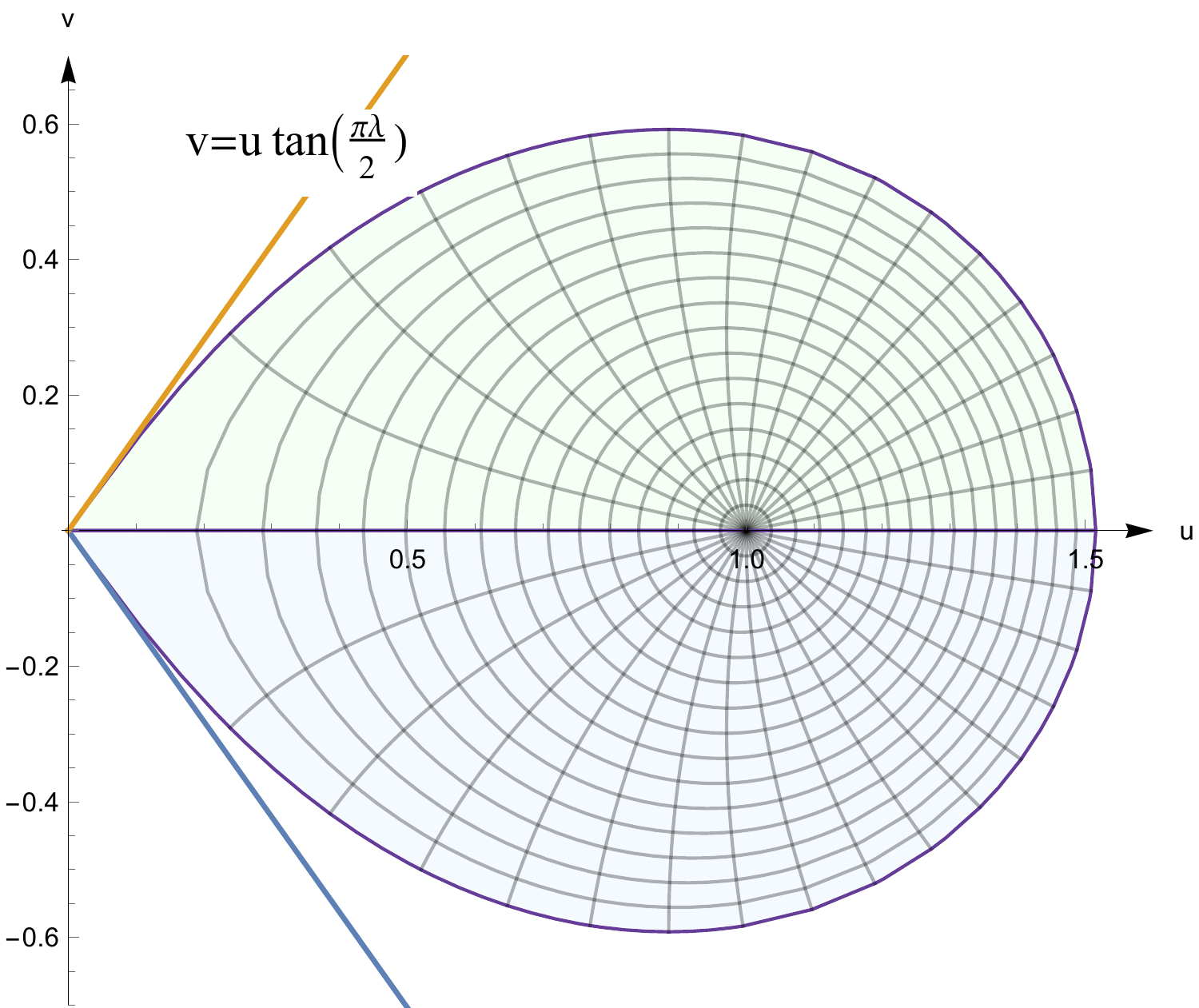}}%
\caption{The image of $\mathfrak{q}_{\lambda}(\mathbb{D})\ (\lambda=0.6)$.}\label{Fig}
\end{figure}
\newline
Now, let us denote
$$\mathfrak{q}_{\lambda}(z):=\myp{1+z}^\lambda= \mathrm{e}^{\lambda\ln(1+z)}\quad\myp{0<\lambda\le1, \,z\in \mathbb{D}},$$
where the branch of the logarithm is determined by $\mathfrak{q}_{\lambda}(0)=1$.
The function $\mathfrak{q}_{\lambda}(z)$ is  convex univalent in $\mathbb{D}$ for each $0<\lambda\le1$ and  maps the unit circle onto a spiral $ \mathbb{SS}(\lambda)$. The power series of $\mathfrak{q}_{\lambda}$ is of the form
\begin{eqnarray}\label{eq_expand_q}
   \mathfrak{q}_{\lambda}(z)&=&1+\sum_{k=1}^{\infty}\dfrac{\lambda\myp{\lambda-1}\cdots\myp{\lambda-k+1}}{k!}
    z^k=1+\sum_{k=1}^{\infty}B_k z^k\quad \myp{z\in \mathbb{D}}.
\end{eqnarray}

\section{Fundamental properties of the family related to $\mathbb{SS}(\lambda)$}\label{fund_prop}
In this section we present some basis results for functions related to the domain $\mathbb{SS}(\lambda)$. Next, we find a relation between $\mathcal{G}(\lambda)$ and  $\mathcal{ST}_{ss}(\lambda)$, that leads to determine the sharp bounds for logarithmic coefficients in $\mathcal{G}(\lambda)$.

\begin{defn}[\cite{MES}]\label{class_ST_ss_lambda}
By $\mathcal{ST}_{ss}(\lambda)$ we denote the subfamily of $\mathcal{ST}$ consisting of the functions $f$ and satisfying the condition
\begin{equation}\label{eq_class_st}
\frac{zf'(z)}{f(z)}\prec \mathfrak{q}_{\lambda}(z) \quad \myp{z\in \mathbb{D}},
\end{equation}
where  $\prec$ denotes the subordination.
\end{defn}
The condition \eqref{eq_class_st} also means that the quantity $zf'(z)/f(z)$ lies in a domain bounded by the spiral $\mathbb{SS}(\lambda)$.
The geometric properties of $\mathfrak{q}_{\lambda}$ imply recurrence inclusions, below.
\begin{equation}\label{inclusion}
\mathcal{ST}_{ss}\left(\frac{\lambda}{n+1}\right)\subset \mathcal{ST}_{ss}\left(\frac{\lambda}{n}\right)\subset\mathcal{ST}_{ss}(\lambda)\quad \left(0<\lambda\le 1,\ n\ge 1\right).\end{equation}
Let $F_{\lambda,n}\in \mathcal{ST}_{ss}(\lambda)$ be given by
\[
\frac{zF'_{\lambda,n}(z)}{F_{\lambda,n}(z)} = \mathfrak{q}_{\lambda}(z^n)=\myp{1+z^n}^\lambda \quad \myp{z\in \mathbb{D}, \, n=1, 2,  \ldots}.
\]
Then, the function $F_{\lambda,n}(z)$ is of the form
 \begin{equation}\label{F_n}
   F_{\lambda,{n}}(z)= z\exp\myp{ \int_{0}^{z} \dfrac{\mathfrak{q}_{\lambda}(t^{n})-1}{t}\,
   \mathrm{d}t} \quad \myp{z\in \mathbb{D}},
\end{equation}
and is extremal for various problems in $\mathcal{ST}_{ss}(\lambda)$. Especially for $n=1$ we have
\begin{equation}\label{F_1}
  F_{\lambda}(z):= F_{\lambda,1}(z)=z\exp\myp{ \int_{0}^{z}
  \dfrac{\mathfrak{q}_{\lambda}(t)-1}{t}\,
   \mathrm{d}t}   =z+\lambda
   z^{2}+\dfrac{3\lambda^2-\lambda}{4}z^3
   +\cdots,
\end{equation}
and setting $\lambda/m$ instead of $\lambda$ with $m,n=1,2,\ldots$,  we obtain the following form
\begin{equation}\label{Fm_n}
F_{\lambda/m,n}(z)= z\exp\myp{ \int_{0}^{z} \dfrac{\mathfrak{q}_{\lambda}(t^n)^{\lambda/m}-1}{t}\,dt} =z+\dfrac{\lambda}{nm}
   z^{n+1}+\dfrac{\lambda^2\myp{n+1}-nm\lambda}{4n^2m^2}z^{2n+1}+\cdots .
\end{equation}
A special case ($n=1$) of $F_{\lambda/m,n}$ gives
 \begin{equation}\label{F_m_1}
F_{\lambda/m}(z)=F_{\lambda/m,1}(z)= z\exp\myp{ \int_{0}^{z} \dfrac{\mathfrak{q}_{\lambda}(t)^{\lambda/m}-1}{t}\,
   \mathrm{d}t}  =z+\dfrac{\lambda}{m}
   z^{2}+\dfrac{2\lambda^2-m\lambda}{4m^2}z^{3}+\cdots \,\, \myp{z\in \mathbb{D}}.
\end{equation}

For $f\in \mathcal{S}$, $zf'(z)/f(z)$ is a non-vanishing analytic function in $\mathbb{D}$. Then, the transform $\mathbf{G}_f$ of $f\in \mathcal{S}$ is well defined, and is of the form \cite{NR}
\begin{equation}\label{eq_G_f}\mathbf{G}_f(z):=\int_{0}^{z}\dfrac{tf'(t)}{f(t)}\,\mathrm{d}t=\int_{0}^{z}\left(1+t\left(\log\dfrac{f(t)}{t}\right)'
 \right)\mathrm{d}t=z+2\sum_{n=2}^{\infty}\dfrac{n-1}{n}\gamma_{n-1}(f)z^{n}.
\end{equation}
Also, let $\mathbf{N}_f$  be given by
\begin{equation}\label{eq_N_f}
\mathbf{N}_f(z):=z\mathbf{G}'_f(z)=z+2\sum_{n=2}^{\infty}(n-1)\gamma_{n-1}(f)z^{n}.
\end{equation}
We observe that $\mathbf{G}_f$ and  $\mathbf{N}_f$ of $f\in \mathcal{S}$ are the functions with the power series expressed in terms of logarithmic coefficients that has the consequences  in the next part of our study.

For functions $F_{\lambda/m,n}$ and $F_{\lambda}$ given by \eqref{Fm_n} and \eqref{F_1} the transform $\mathbf{G}$ and  $\mathbf{N}$ yield
\begin{align}
  \mathbf{G}_{F_{\lambda/m,n}}(z)&=\int_{0}^{z}\mathfrak{q}_{\lambda/m}\left(t^n\right) \mathrm{d}t=z+\dfrac{\lambda}{m(n+1)}z^{n+1}+\dfrac{\lambda(\lambda-m)}{2m^2(2n+1)}z^{2n+1}+\cdots.\label{EQ_G}\\
 \mathbf{G}_{F_{\lambda}}(z)&= \mathbf{G}_{F_{\lambda/1,1}}(z)=\dfrac{(1+z)^{1+\lambda}-1}{1+\lambda}=z+\dfrac{\lambda}{2}z^{2}+\frac{\lambda(\lambda-1)}{6}z^3+\cdots.\label{EQ_G_1}\\
  \mathbf{N}_{F_{\lambda/m,n}}(z)&=z\mathfrak{q}_{\lambda/m}\left(z^n\right)=z+\dfrac{\lambda}{m} z^{n+1}+\frac{\lambda(\lambda-m)}{2m^2}z^{2n+1}+\cdots,\label{EQ_N}
 \end{align}
 and $\mathbf{N}_{F_{\lambda}}(z)=\mathbf{N}_{F_{\lambda/1,1}}(z)$.
The sample figures of $\mathbf{N}_{F_{\lambda}}$ and  $ \mathbf{G}_{F_{\lambda}}$  are presented on  Fig. \ref{Fig1} and Fig \ref{Fig2}.\\

Let $m,n=1,2,\ldots$ and $0<\lambda\le 1$. Then, the logarithmic coefficients of $\mathbf{G}_{F_{\lambda/n,n}}$ and $F_{\lambda,n}$ are the following
\begin{equation*}
\gamma_n\myp{\mathbf{G}_{F_{\lambda/n,n}}}=\frac{\lambda}{2n(n+1)},\quad\quad
\gamma_n\myp{F_{\lambda,n}}=\frac{\lambda}{2n},\quad\quad \gamma_n\myp{F_{\lambda}}=\frac{B_n}{2n}.
\end{equation*}
Also, for $m\ge n$, we have
\[\mathbf{G}_{F_{\lambda/m,n}}\in \mathcal{G}(\lambda),\qquad	\mathbf{N}_{F_{\lambda/m,n}}\in \mathcal{N}(\lambda).\]

  \begin{figure}[!h]
\centering
\subfloat[   $\mathbf{N}_{F_{\lambda}}(z)$]{%
\includegraphics[width=0.3\textwidth]{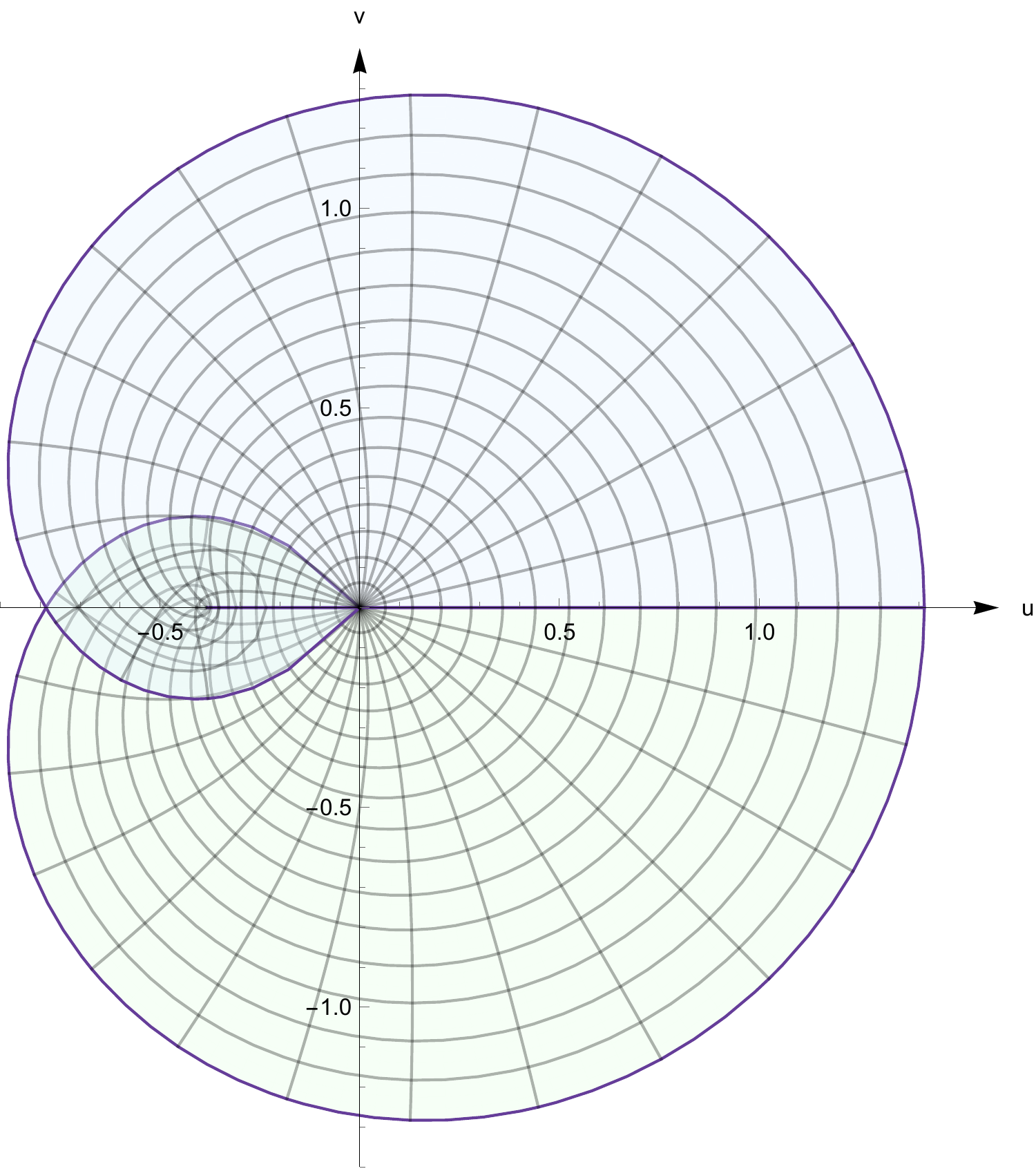}}%
\qquad\qquad
\subfloat[$\mathbf{G}_{F_{\lambda}}(z)$ ]{%
\includegraphics[width=0.3 \textwidth]{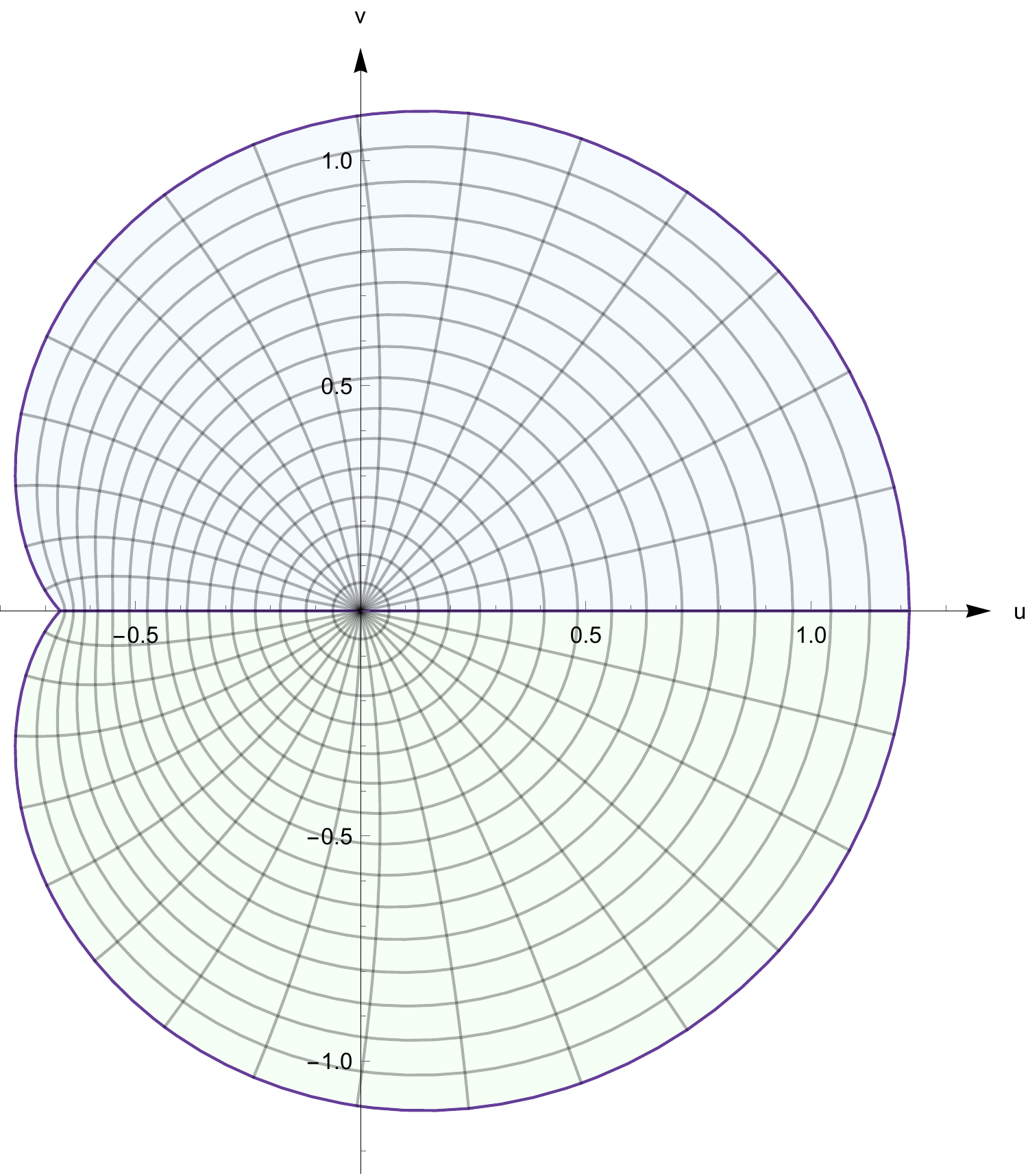}}%
\caption{The image of $ \mathbf{N}_{F_{\lambda}}(z), \ \mathbf{G}_{F_{\lambda}}(z), \ (\lambda=1/2)$.}\label{Fig1}
\end{figure}
  \begin{figure}[!h]
\centering
\subfloat[$\mathbf{N}_{F_{\lambda,n}}(z)$]{%
\includegraphics[width=0.35\textwidth]{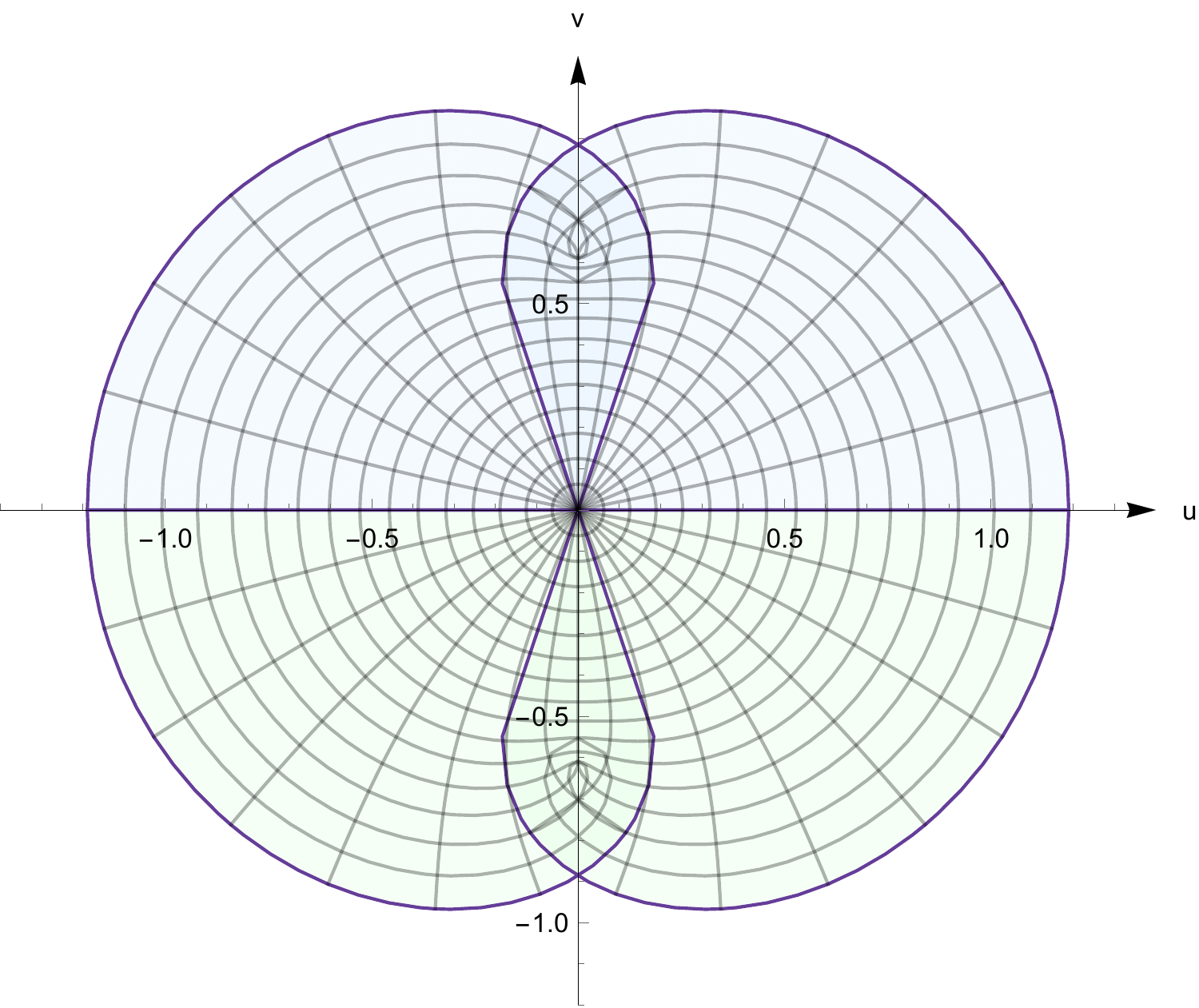}}%
\qquad
\subfloat[$\mathbf{G}'_{F_{\lambda,n}}(z)=\mathfrak{q}_\lambda\left(z^n\right)$ ]{%
\includegraphics[width=0.4 \textwidth]{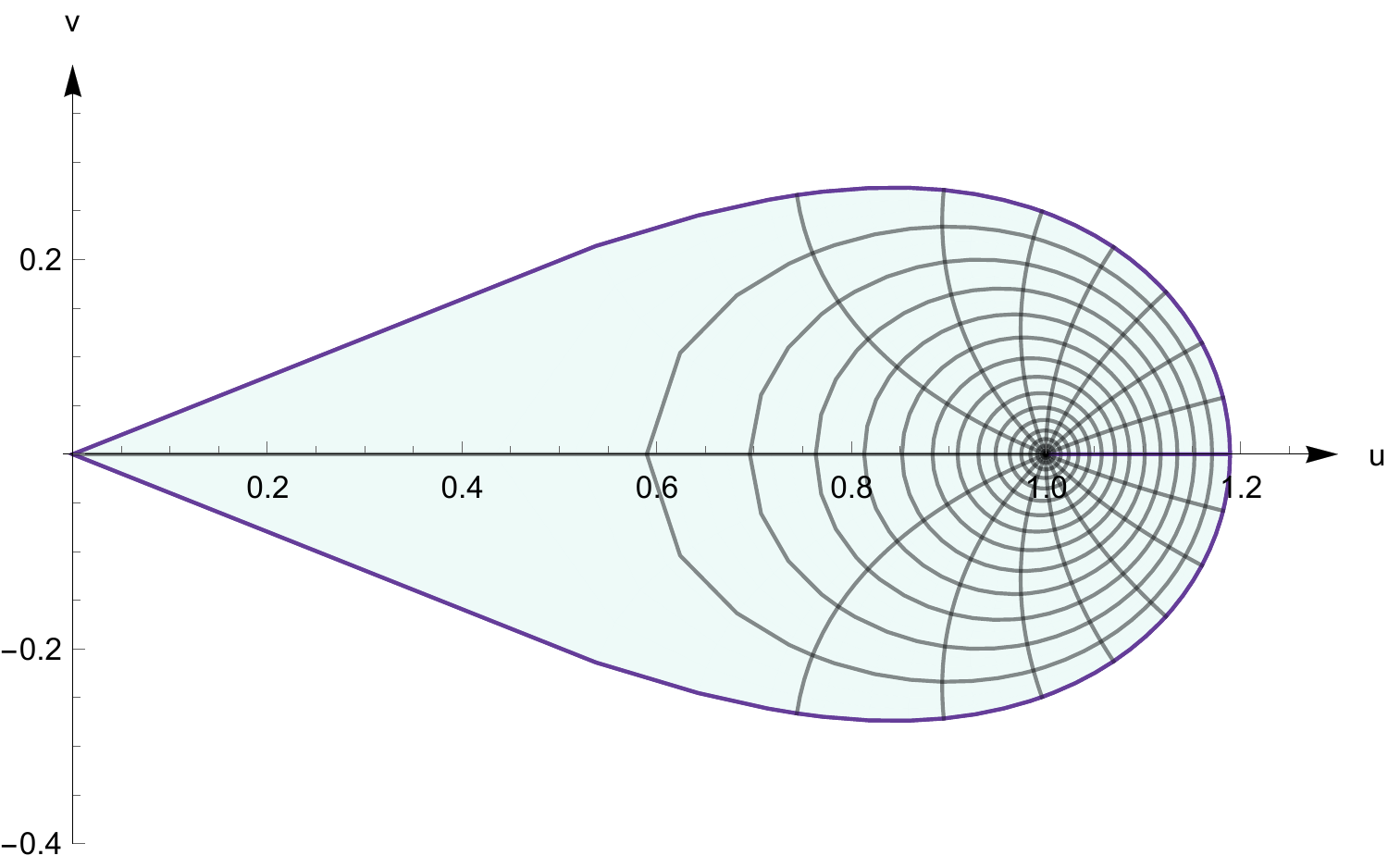}}%
\caption{The image of $ \mathbf{N}_{F_{\lambda,n}}(z), \ \mathbf{G}'_{F_{\lambda,n}}(z), \ (\lambda=1/2,\ n=2)$.}\label{Fig2}
\end{figure}

\section{Relations between $\mathcal{G}(\lambda)$, $\mathcal{N}(\lambda)$ and $\mathcal{ST}_{ss}(\lambda)$}\label{sec_P_R}
In this section, we study the relationship between  the classes $\mathcal{G}(\lambda)$, $\mathcal{N}(\lambda)$ and $\mathcal{ST}_{ss}(\lambda)$, and determine the theorems of growth, distortion, and rotation in the families $\mathcal{G}(\lambda)$ and  $\mathcal{N}(\lambda)$. By discovering such a new relation we are able to confirm the truth of the inequality conjectured in \cite{Pon}.

In order to prove our main results we will use some useful  following lemmas concerning subordination.\\
Let us denote by $\mathcal{Q}$  the class of functions $f$ that are
analytic and injective on $\overline{\mathbb{D}}\setminus \mathbf{E}(f)$, where
$\mathbf{E}(f) =\left\{\zeta \colon \zeta \in \partial\mathbb{D}, \
\lim_{z\to \zeta}f(z)=\infty\right\},
$
and are such that
\begin{equation*}
f'(\zeta)\neq 0\quad \textit{for}\quad \zeta\in \partial\mathbb{D}\setminus \mathbf{E}(f).
\end{equation*}
\begin{lem}\cite[Theorem 2.2d, p.24]{MM}\label{lem_1}
Let $q\in \mathcal{Q}$ with $q(0)=1$ and let $p(z)=1+p_nz^n+\cdots$ be analytic in $\mathbb{D}$ with $p(z)\neq 1$.
If $p\not\prec q$ in $\mathbb{D}$, then there exits points $z_0\in \mathbb{D}$ and $\zeta\in
\partial\mathbb{D}\setminus \mathbf{E}(q)$
and there exits a real number $m\ge n\ge  1$ for which
 \begin{equation*}
p\myp{|z|<|z_0|}\subset q(\mathbb{D}),\quad p(z_0)=q(\zeta), \quad z_0p'(z_0) =m
\zeta q'(\zeta).
\end{equation*}
\end{lem}
\begin{lem}\cite[Theorem 3.1b, p.71]{MM}\label{lem_2}
Let $h(z)$ be convex in $\mathbb{D}$ with $h(0)=a$. If $p(z)$ is analytic in $\mathbb{D}$, with $p(0)=a$ and $p(z)+zp'(z)\prec h(z)$, then
\[
p(z)\prec \dfrac{1}{z}\int_{0}^{z} h(t)\, \mathrm{d}t.
\]
\end{lem}
\begin{lem}\cite[Corollary 3.1d.1, p.76]{MM}\label{lem_3}
Let $h(z)$ be starlike in $\mathbb{D}$ with $h(0)=0$ and $a\ne0$. If $p(z)=1+p_nz^n+p_{n+1}z^{n+1}+\cdots$ is analytic in $\mathbb{D}$ satisfies the subordination relation
\[
\frac{zp'(z)}{p(z)} \prec h(z)\quad \Longrightarrow p(z)\prec a\exp\left(\frac{1}{n}\int_{0}^{z} \frac{h(t)}{t}\, \mathrm{d}t\right).
\]
\end{lem}
\begin{lem}[\cite{Libera}]\label{Libera}
Let $f\in \mathcal{CV}$. If $g(z)=(1/z)\int_{0}^{z}f(t)\, \mathrm{d}t$, then $g$ is also univalent and convex in $\mathbb{D}$.
\end{lem}

\begin{thm}\label{p_q}
	Let $p(z)=1+p_1z+\cdots$ be  analytic function in the unit disk $\mathbb{D}$ with $p(z)\not\equiv 1$. If
	\begin{equation}\label{p'}
	\Re\left\{\frac{zp'(z)}{p(z)}\right\}<\frac{\lambda}{2}\qquad \myp{0<\uplambda\le1,\, z\in \mathbb{D}},
	\end{equation}
	then there exists $n\ge 1$ such that 
	\[
p(z)\prec\myp{1+z^n}^{\lambda/n}=:q(z)\quad \myp{z\in \mathbb{D}}.
	\]
\end{thm}
\begin{proof}
	Since $p(z)\not\equiv 1$, thus there exists $k\ge 1$ such that $p_k\neq 0$, where $p_k$ is the $kn$-th coefficient in the expansion of the function $p(z)$. Suppose, on the contrary, that $p(z)\not\prec q(z)$ on $\mathbb{D}$. Then by Lemma \ref{lem_1} there exist $z_0\in \mathbb{D}$ and $\zeta_0 \in\partial\mathbb{D}$ with  $\zeta_0\neq-1$ such that
	\begin{equation*}
	p(z_0)=q(\zeta_0),\quad  z_0p'(z_0)= m \zeta_0 q'(\zeta_0)\quad m \ge 1.
	\end{equation*}
	Thus
	\[
	\Re\left\{\frac{z_0p'(z_0)}{p(z_0)}\right\}=\Re\left\{\frac{m\zeta_0q'(\zeta_0)}{q(\zeta_0)}\right\}=m\lambda\Re\left\{\frac{\zeta_0^n}{1+\zeta_0^n}\right\}=\frac{m\lambda}{2}\ge \frac{\lambda}{2}.
	\]
	But the above contradicts the assumption \eqref{p'} and therefore $p\prec q$ on $\mathbb{D}$ follows.
\end{proof}
Letting $p(z)=zf'(z)/f(z)$ in Theorem \ref{p_q}, we obtain the  following corollary.

\begin{cor}\label{cor_G_lambda_starlike}
	If a functions $f$  defined by \eqref{eq_expand_f} satisfies the following condition
\begin{equation}\label{Cor 3.6}
	 \Re\left\{1+\frac{zf''(z)}{f'(z)}-\frac{zf'(z)}{f(z)} \right\}<\frac{\lambda}{2}\quad \left(z\in \mathbb{D}\right),
\end{equation}
then there exists $n\ge 1$ such that  $zf'(z)/f(z)\prec \myp{1+z^n}^{\lambda/n}$ in $\mathbb{D}$.
\end{cor}
We note that 
$$1+\dfrac{z\mathbf{G}_{f}''(z)}{\mathbf{G}_{f}'(z)}=1+\frac{zf''(z)}{f'(z)}-\frac{zf'(z)}{f(z)},$$
and
$$\dfrac{z\mathbf{N}_{f}'(z)}{\mathbf{N}_{f}(z)}=1+\frac{zf''(z)}{f'(z)}-\frac{zf'(z)}{f(z)},$$
therefore the Corollary \ref{cor_G_lambda_starlike} may be rewritten as 
\begin{cor}\label{cor_condition_uni}
	Let $f\in\mathcal{A}$. If $\mathbf{G}_{f}\in \mathcal{G}(\lambda)$ or  $\mathbf{N}_{f}\in \mathcal{N}(\lambda)$,  then there exists $n\ge 1$ such that   $f\in \mathcal{ST}_{ss}(\lambda/n)\subset \mathcal{ST}$.
\end{cor}
Set now  $p(z)=f'(z)$  in Theorem \ref{p_q}. Then the condition \eqref{p'} is equivalent to the fact $\Re\,( zf''(z)/f'(z)) < \lambda/2$ or $\Re\, (1+zf''(z)/f'(z)) < 1+\lambda/2$ that defines functions from the class $\mathcal{G}(\lambda)$. This observation and Theorem \ref{p_q} gives the following integral representation in $\mathcal{G}(\lambda)$.

\begin{cor}\label{cor_G_lambda}
	If a functions $f$  defined by \eqref{eq_expand_f} belongs to the class $\mathcal{G}(\lambda)$, then there exists $n\ge 1$ such that  $f'\prec
	\myp{1+z^n}^{\lambda/n}$ in $\mathbb{D}$, and
\begin{equation}\label{eq_cor_G_lambda}
	g(z)=z\exp\myp{\int_{0}^{z}\dfrac{f'(t)-1}{t}\,\mathrm{d}t}\in \mathcal{ST}_{ss}(\lambda/n).
\end{equation}
\end{cor}
The  class $\mathcal{ST}_{ss}(\lambda)$ is  a subclass of the class of starlike functions, hence by the corollary \ref{cor_G_lambda}, we can obtain a representation in the class $\mathcal{G}(\lambda)$. Namely, $f\in \mathcal{G}(\lambda)$ if and only there exists a function $g\in\mathcal{ST}_{ss}(\lambda) $  such that
\begin{equation}
f(z)=\mathbf{G}_g(z)\quad \left(z\in \mathbb{D} \right)
\end{equation}
where $\mathbf{G}_g$ is given by \eqref{eq_G_f}.

Setting  now  $p(z)=f(z)/z$ we have $zp'(z)/p(z) = zf'(z)/f(z)-1$ and  Theorem \ref{p_q} gives:
\begin{cor}\label{cor_N_lambda}
If a functions $f$ defined by \eqref{eq_expand_f} belongs to the class $\mathcal{N}(\lambda)$, then there exists $n\ge 1$ such that    $f(z)/z\prec \myp{1+z^n}^{\lambda/n}$ in $\mathbb{D}$ and
\begin{equation}\label{eq_cor_N_lambda}
	g(z)=z\exp\myp{\int_{0}^{z}\frac{f(t)-t}{t^2}\,\mathrm{d}t}\in \mathcal{ST}_{ss}(\lambda/n).
\end{equation}
\end{cor}
Alternatively, $f\in \mathcal{N}(\lambda)$ if and only there exists a function $g\in\mathcal{ST}_{ss}(\lambda) $  such that
\begin{equation}
f(z)=\mathbf{N}_g(z)\quad \left(z\in \mathbb{D} \right)
\end{equation}
where $\mathbf{N}_g$ given by \eqref{eq_N_f}.
\begin{thm}\label{lem_f_z}
	If a function $f$ with the power series \eqref{eq_expand_f} belongs to the class $\mathcal{G}(\uplambda)$, then
	\begin{equation}\label{eq_lem_f_z}
	\dfrac{f(z)}{z}\prec \dfrac{1}{z}\int_{0}^{z}(1+t^n)^{\lambda/n}\, \mathrm{d}t=\dfrac{\mathbf{G}_{F_{\lambda/n,n}}(z)}{z}\quad \myp{n\ge 1,\ z\in \mathbb{D}}.
	\end{equation}
	where $\mathbf{G}_{F_{\lambda/n,n}}$ given by \eqref{EQ_G}. Especially for $n=1$, we have
	\begin{equation}\label{eq_lem_f_z_1}
	\dfrac{f(z)}{z}\prec \dfrac{\mathbf{G}_{F_{\lambda}}(z)}{z}\quad \myp{z\in \mathbb{D}},
	\end{equation}
	where $\mathbf{G}_{F_{\lambda}}$   given by \eqref{EQ_G_1}.
\end{thm}
\begin{proof} Let $h(z)=f(z)/z$. Then $h(z)+zh'(z)=f'(z)$. Application of the corollary \ref{cor_G_lambda} enables us to write
	\begin{equation}\label{}
	h(z)+zh'(z)=f'(z)\prec (1+z^n)^{\lambda/n}\quad \left(n\ge1,\ z\in \mathbb{D} \right),
	\end{equation}
and use of Lemma \ref{lem_2} to the above relation, gives
	\[
	h(z)=\dfrac{f(z)}{z}\prec \dfrac{1}{z}\int_{0}^{z}(1+t^n)^{\lambda/n}\, \mathrm{d}t.
	\]
\end{proof}
\begin{thm}\label{mathbb{H}_F_lambda_n}
	The function $\frac{\mathbf{G}_{F_{\lambda/n,n}}(z)}{z}$ is  convex univalent in $\mathbb{D}$ for each
	$0<\uplambda\le1$ and $n\ge 1$  (not normalized in the usual sense).
	Especially $\frac{\mathbf{G}_{F_{\lambda}}(z)}{z}$ is  convex univalent in $\mathbb{D}$, and   
	\[
	\frac{2}{\lambda}\left(\frac{\mathbf{G}_{F_{\lambda}}(z)}{z}-1\right)
	\in \mathcal{CV}.\]
\end{thm}
\begin{proof}
The function  $\left(1+z^n\right)^{\lambda/n}$ for each $0<\lambda\le 1$ and $n\ge 1$  is univalent and  convex. Taking into account Lemma \ref{Libera}, the functions   $\frac{\mathbf{G}_{F_{\lambda/n,n}}(z)}{z}$ are convex univalent in $\mathbb{D}$. The image of the unit disk by  $\frac{\mathbf{G}_{F_{\lambda}}(z)}{z}$ and  $\log\frac{\mathbf{G}_{F_{\lambda}}(z)}{z}$ are represented in the figure Fig. \ref{log_G}.
\end{proof}
\begin{figure}[!h]
\centering
\subfloat[   $\frac{\mathbf{G}_{F_{\lambda}}(z)}{z}$]{%
\includegraphics[width=0.4\textwidth]{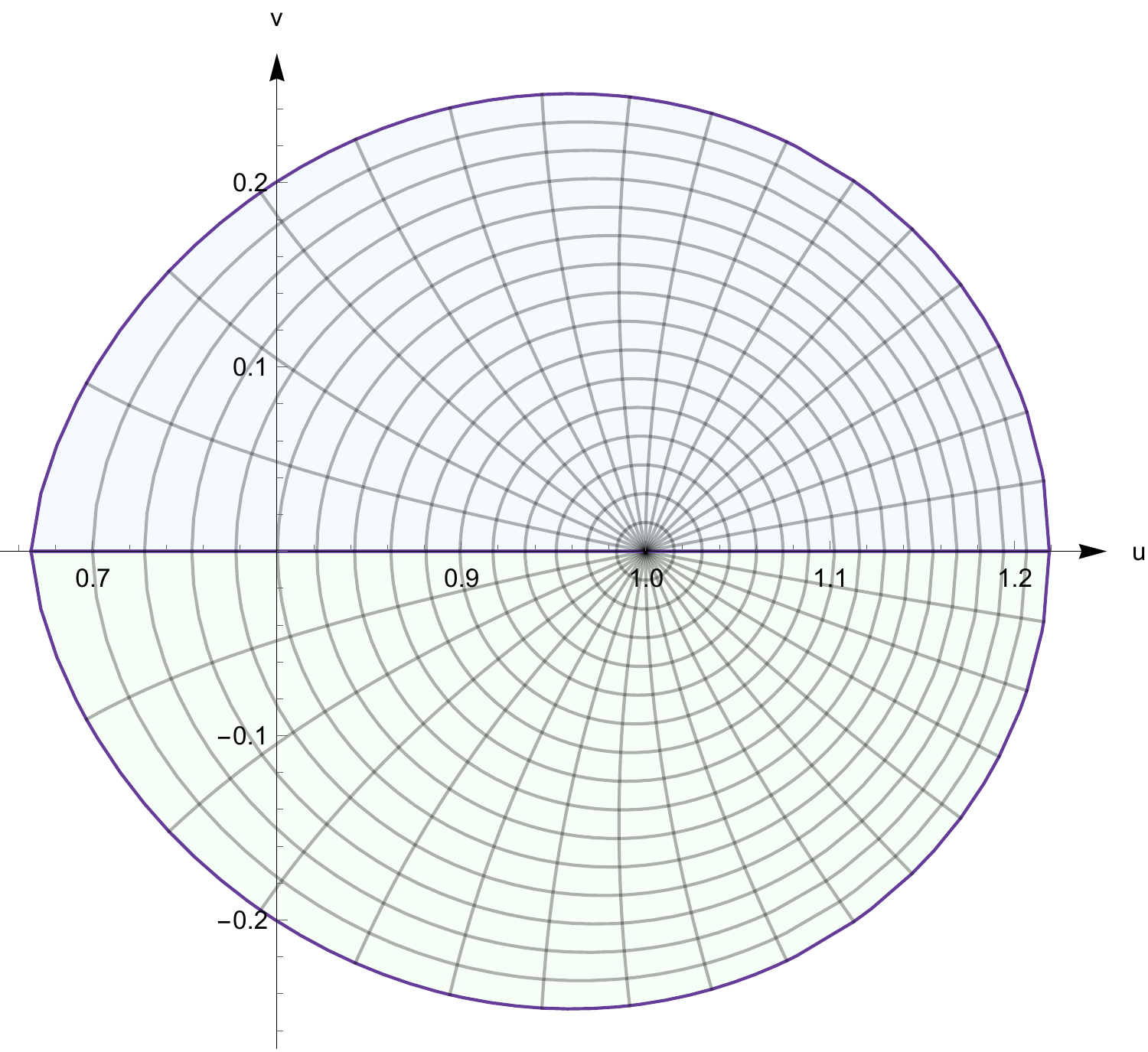}}%
\hfill
\subfloat[$\log\frac{\mathbf{G}_{F_{\lambda}}(z)}{z}$ ]{%
\includegraphics[width=0.43 \textwidth]{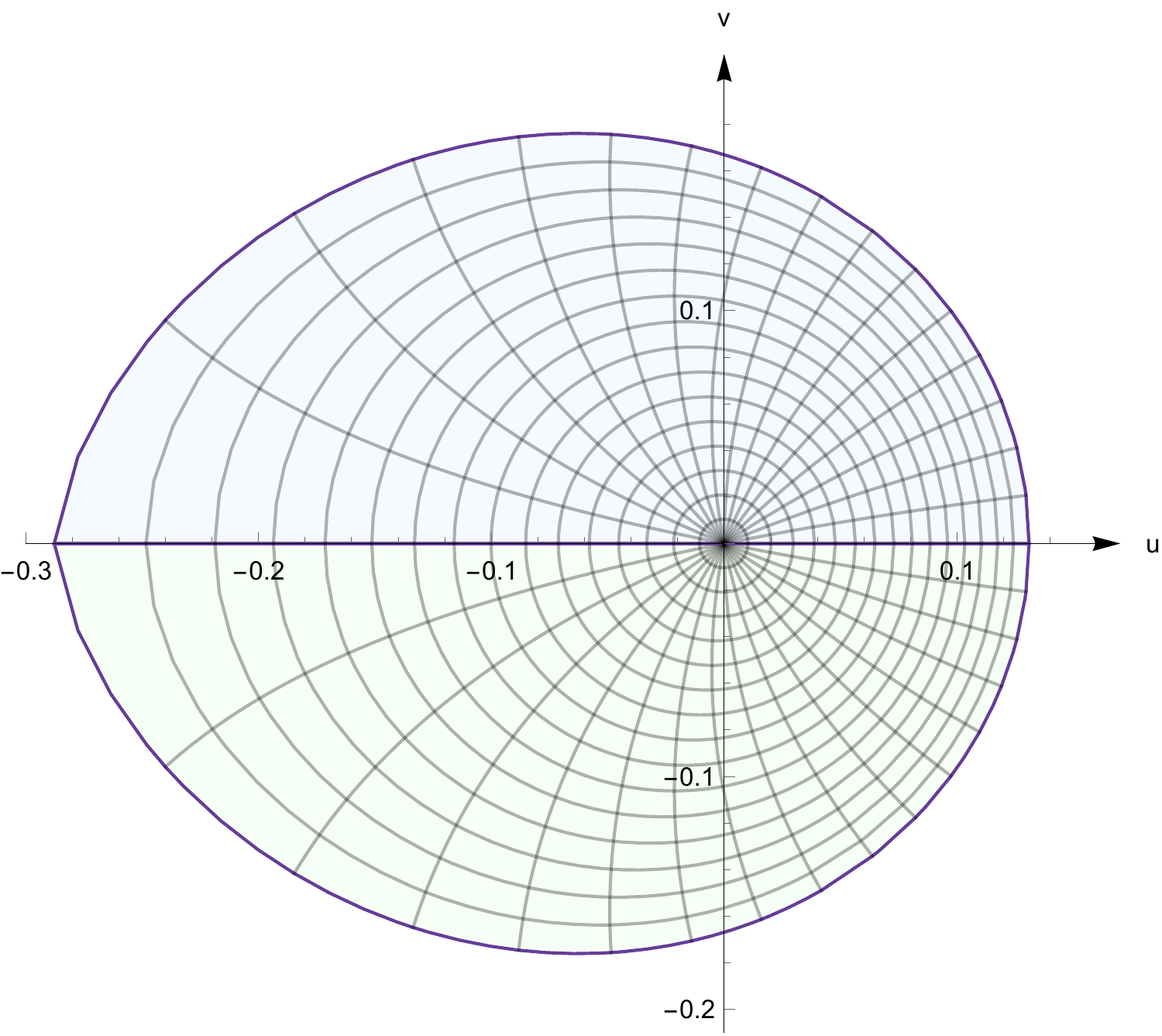}}%
\caption{The image of $\frac{\mathbf{G}_{F_{\lambda}}(z)}{z}$ and $\log\frac{\mathbf{G}_{F_{\lambda}}(z)}{z} \ (\lambda=1/2)$.}\label{log_G}
\end{figure}
Applying Corollary \ref{cor_N_lambda}, we immediately obtain.
\begin{thm}
	Let $f\in \mathcal{N}(\lambda)$ given by \eqref{eq_expand_f}. Then for $|z|=r<1$ the following inequalities hold true
	\begin{enumerate}
		\item[(i)]
		$
		- \mathbf{N}_{F_{\lambda}}(-r)=r(1-r)^\lambda\le \left|f(z)\right|\le r(1+r)^\lambda= \mathbf{N}_{F_{\lambda}}(r),
		$
		\item[(ii)]
		$
		\left|{\rm arg}\left\{\dfrac{f(z)}{z}\right\}\right|\le \lambda \sin^{-1} r.
		$
	\end{enumerate}
	Equalities hold for the function $f(z)= \mathbf{N}_{F_{\lambda}}(z)$ given by \eqref{EQ_N} or its rotation.
\end{thm}

\begin{thm}
	Let $f\in \mathcal{G}(\lambda)$ given by \eqref{eq_expand_f}. Then for $|z|=r<1$ it holds
	\begin{enumerate}
		\item[(i)]
		$\mathfrak{q}_{\lambda}(-r)= (1-r)^\lambda\le \left|f'(z)\right|\le (1+r)^\lambda=\mathfrak{q}_{\lambda}(r)$,
		\item[(ii)]
		$\left|{\rm arg}\{f'(z)\}\right|\le \lambda \sin^{-1} r$,
		\item[(iii)]
		$-\mathbf{G}_{F_{\lambda}}(-r)=\dfrac{1-(1-r)^{\lambda+1}}{1+\lambda}\le \left|f(z)\right|\le \dfrac{(1+r)^{\lambda+1}-1}{1+\lambda}=\mathbf{G}_{F_{\lambda}}(r).$
	\end{enumerate}
	The above inequalities are sharp; equalities hold for $f(z)=\mathbf{G}_{F_\lambda}(z)$ given by \eqref{EQ_G_1} or its rotation.
\end{thm}
\begin{proof}
	Proofs of ${\rm (i)}$ and ${\rm (ii)}$ easily follow from corollary \ref{cor_G_lambda} and ${\rm(iii)}$
	follows from Theorem \ref{lem_f_z}.
\end{proof}
\section{Logarithmic coefficients in the  class $\mathcal{G}(\lambda)$}\label{sec_log}
Based on the results of the previous section we obtain a sharp bounds of logarithmic coefficient in the  class $\mathcal{G}(\lambda)$.
\begin{thm}\label{thm_log_coff}
Let $f\in\mathcal{ST}_{ss}(\lambda)$. Then the logarithmic coefficients of $f$ satisfy
\begin{eqnarray}
    \left|\gamma_n(f)\right|&\le& \dfrac{\lambda}{2n} \quad \myp{n\geq 1},\label{eq1_thm_log_coff}\\
         \sum_{n=1}^{\infty}n^ 2  \left|\gamma_n(f)\right|^2 &\le & \frac14\sum_{n=1}^{\infty}\left|B_n\right|^2,\label{eq2_thm_log_coff}\\
    \sum_{n=1}^{\infty}  \left|\gamma_n(f)\right|^2 &\le & \frac14\sum_{n=1}^{\infty}\frac{\left|B_n\right|^2}{n^2}\le \frac{\lambda^2 \pi^2}{24}.\label{eq3_thm_log_coff}
    \end{eqnarray}
 All the inequalities are sharp.
\end{thm}
\begin{proof}
Let $f\in \mathcal{ST}_{ss}(\lambda)$.  Then the relation \eqref{eq_class_st} can be rewritten  as

\begin{equation}\label{logf/z}
    z\myp{\log \dfrac{f(z)}{z}}' \prec \mathfrak{q}_{\lambda}(z)-1\quad \myp{z\in \mathbb{D}}
\end{equation}
which, in terms of the coefficients, is equivalent to
\begin{equation}\label{eq_thm_log_coff_1}
    \sum_{n=1}^{\infty}2n\gamma_n(f) z^n\prec \sum_{n=1}^{\infty}B_n  z^n.
\end{equation}
By Rogosinski result for coefficients of subordinate functions \cite{Rog}, we obtain $2n\left|\gamma_n(f)\right|\leq
|B_1|=\lambda$ from which \eqref{logf/z} follows. The result is sharp for $F_{\lambda,n}(z)$ given by \eqref{F_n} for that we set $zF'_{\lambda,n}(z)/F_{\lambda,n}(z)=  \mathfrak{q}_{\lambda}(z^n)$.

In order to prove \eqref{eq2_thm_log_coff}, let $f\in \mathcal{ST}_{ss}(\lambda)$. Applying repeatadly Rogosinski’s theorem \cite{Rog} on the  subordination relation, and \eqref{eq_thm_log_coff_1} we obtain
\[
4\sum_{m=1}^{N} n^2 |\gamma_{n}(f)|^2 \le \sum_{n=1}^{N}  \left|B_n\right|^2\le \sum_{n=1}^{\infty}\left|B_n\right|^2,
\]
which proves the assertion \eqref{eq2_thm_log_coff} if we allow $N\to \infty$. The inequality \eqref{eq2_thm_log_coff} is sharp for the function $f(z)=\overline{\mu}F_{\lambda,n}(\mu z)$, given by \eqref{F_n}, where $\mu$ is a unimodular complex number.

Denote now  $g(z) := f (z)/z$. Function $g$ is  analytic in $\mathbb{D}$ and such that $g(0) = 1$. Moreover, by \eqref{logf/z}, $g$ satisfies the relation
\[
\frac{z g^{\prime}(z)}{g(z)}= z\myp{\log \dfrac{f(z)}{z}}'=\frac{z f^{\prime}(z)}{f(z)} -1\prec \mathfrak{q}_{\lambda}(z)-1 \quad \left(z \in\mathbb{D}\right).
\]
Since $\mathfrak{q}_{\lambda}(z)-1$ is convex in $\mathbb{D}$ and $\mathfrak{q}_{\lambda}(0)-1=0$ then, using Lemma \ref{lem_3} we conclude
\[
g(z) = \frac{f(z)}{z}\prec \exp\left(\int_{0}^{z}\frac{\mathfrak{q}_{\lambda}(t)-1}{t}\, \mathrm{d}t \right).
\]
Then $\log (f(z)/z)\prec \int_{0}^{z}[(\mathfrak{q}_{\lambda}(t)-1)/t]\, dt$, that is  equivalent to
\begin{equation}\label{eq5_thm_log_coff}
\sum_{n=1}^{\infty}2\gamma_n(f) z^n \prec \int_{0}^{z}\frac{\mathfrak{q}_{\lambda}(t)-1}{t}\, \mathrm{d}t=\sum_{n=1}^{\infty}\frac{B_n}{n}z^n.
\end{equation}
Application of the Rogosinski theorem \cite{Rog}  gives
\[
4\sum_{n=1}^{N}  |\gamma_{n}(f)|^2 \prec \sum_{n=1}^{N}  \frac{\left|B_n\right|^2}{n^2}\le \sum_{n=1}^{\infty}  \frac{\left|B_n\right|^2}{n^2},
\]
which proves the desired assertion \eqref{eq3_thm_log_coff} when $N\to \infty$. The inequality \eqref{eq3_thm_log_coff} is sharp for the function $f(z)=\overline{\mu}F_{\lambda}(\mu z)$ given by \eqref{F_1}, where $\mu$ is a unimodular complex number.
\end{proof}
Application of  the above theorem with $\lambda/m$ instead of $\lambda$ leads to the following. 
\begin{cor}\label{cor_log_coff_n}
The logarithmic coefficients of $f\in\mathcal{ST}_{ss}(\lambda/m)$ satisfy sharp inequalities
\begin{equation*}
    \left|\gamma_n(f)\right|\leq \dfrac{\lambda}{2nm} \quad \myp{n\geq 1}.
\end{equation*}
The functions $f(z)=\overline{\mu}F_{\lambda/m,n}(\mu z)$, where $\mu$ is a unimodular complex number, realizes equality.
\end{cor}
\begin{thm}\label{thm_log_coff_G}
The logarithmic coefficients of $f\in\mathcal{G}(\lambda)$	satisfy sharp inequalities
	\begin{equation}\label{log_G1}
		\left|\gamma_{n}(f)\right| \le
		\dfrac{\lambda}{2n(n+1)}\quad  \myp{n= 1,2 ,3 ,\ldots}.
	\end{equation}
The inequality is sharp.
\end{thm}
\begin{proof}
	Let $f\in \mathcal{G}(\lambda)$. Then, the relation  \eqref{eq_lem_f_z} is equivalent to
\begin{equation}\label{EQ_Log_G2}
	\log\dfrac{f(z)}{z}\prec \log\dfrac{\mathbf{G}_{F_{\lambda/n,n}}(z)}{z}\quad \myp{n\ge 1,\ z\in \mathbb{D}}.
\end{equation}
In terms of  the logarithmic coefficients  the previous relation is equivalent to
\[
\sum_{m=1}^{\infty} 2 \gamma_{m}(f) z^{m} \prec
\sum_{m=1}^{\infty} 2 \gamma_{m}\left(n,\mathbf{G}_{F_{\lambda/n,n}}\right) z^{m}=2 \gamma_{n}\left(n,\mathbf{G}_{F_{\lambda/n,n}}\right) z^{n}+\cdots\quad(z \in \mathbb{D}).
\]
A simple computation shows that  $\gamma_{n}\left(n,\mathbf{G}_{F_{\lambda/n,n}}\right)=\lambda/(2n(n+1))$.
Since  the class of convex univalent functions is closed under convolution \cite{RSS}  then, making use of Corollary \ref{cor_condition_uni} and  Theorem \ref{lem_f_z}, the function
\[
\log\frac{\mathbf{G}_{F_{\lambda/n,n}}(z)}{z}=\int_{0}^{z} \dfrac{\mathfrak{q}_{\lambda}(t^n)^{\lambda/n}-1}{t}\,\mathrm{d}t
\]
is convex and univalent in $\mathbb{D}$. Applying now Rogosinski result and \eqref{EQ_Log_G2} we obtain
\[
2\left|\gamma_m(f)\right|\le \dfrac{\lambda}{n(n+1)}\quad \myp{m=1,2, \ldots}.
\]
The above  inequality is sharp for the function $f(z)=\mathbf{G}_{F_{\lambda/n,n}}(z)$ given by \eqref{EQ_G} for which 
$2\gamma_{n}\left(n,\mathbf{G}_{F_{\lambda/n,n}}\right)=\lambda/(n(n+1))$.
\end{proof}	
We note that for $N=1,2,\ldots$
$$\sum_{n=1}^{N}\dfrac{1}{n(n+1)}=1-\dfrac{1}{N+1},$$ 
and
$$\sum_{n=1}^{N}\dfrac{1}{n^2(n+1)^2}=\sum_{n=1}^{N}\frac{1}{n^2}+\sum_{n=1}^{N}\frac{1}{(n+1)^2}+2\sum_{n=1}^{N}\left(\frac{1}{n+1}-\frac{1}{n}\right),$$
therefore as the corollaries of Theorem \ref{thm_log_coff_G}, we present the following inequalities.
\begin{cor}
The logarithmic coefficients of $f\in\mathcal{G}(\lambda)$ satisfy
\begin{eqnarray*}
  \sum_{n=1}^{\infty} \left|\gamma_n(f)\right|&\le& \dfrac{\lambda}{2},\\
         \sum_{n=1}^{\infty}n^ 2  \left|\gamma_n(f)\right|^2 &\le & \frac{\lambda^2}{24}\left(\pi^2-6\right),\\
         \sum_{n=1}^{\infty}(n+1)^ 2  \left|\gamma_n(f)\right|^2 &\le & \frac{\lambda^2}{24}\pi^2,\\
    \sum_{n=1}^{\infty}  \left|\gamma_n(f)\right|^2 &\le & \frac{\lambda^2}{12}\left(\pi^2-9\right).
    \end{eqnarray*}
\end{cor}

Below, we present the another way to obtain estimate of the coefficient bounds of the functions in families $\mathcal{G}(\lambda)$ and $\mathcal{N}(\lambda)$.
\begin{thm}\label{theoremCV}
	Let the function $f\in\mathcal{G}(\lambda)$ be given by \eqref{eq_expand_f}.  Then   $|a_n|\le \lambda/(n(n-1))\ (n\ge 2)$. Equality holds for $\mathbf{G}_{F_{\lambda/(n-1),n-1}}$ defined by \eqref{EQ_G}.
\end{thm}
\begin{proof}
	Let $f\in\mathcal{G}(\lambda)$ be given by \eqref{eq_expand_f}. Then, by Corollary \ref{cor_G_lambda}, there exists $n\ge 1$ and $h\in \mathcal{ST}_{ss}(\lambda/n)$ such that
	\[
	f(z)=\mathbf{G}_{g}(z)\quad \left(z\in \mathbb{D}\right),
	\]
that is equivalent to
	\[
	\sum_{n=2}^{\infty}a_nz^n=2\sum_{n=2}^{\infty}\dfrac{n-1}{n}\gamma_{n-1}(g)z^{n},
	\]
	where $\gamma_n(g)$ is a logarithmic coefficients of $g$. By Corollary \ref{cor_log_coff_n} we have $\left|\gamma_{n-1}(g)\right|\le \lambda /(n-1)^2$, and then we conclude $|a_n|\le \lambda/(n(n-1))$. The  inequality is sharp for  $g=F_{\lambda/(n-1),n-1}$ or alternatively $f=\mathbf{G}_{F_{\lambda/(n-1),n-1}}$. Note that for function $\mathbf{G}_{\lambda}$ given by \eqref{EQ_G_1}, we have
	\[
	a_n=\frac{\lambda(\lambda-1)\cdots (\lambda-n+2)}{n!}\quad \left(n\ge 2\right),
	\]
and so $|a_n|\le \lambda/(n(n-1))$ is satisfied.	
\end{proof}
\begin{thm}
	If a function $f\in\mathcal{N}(\lambda)$ is of the form  \eqref{eq_expand_f} 
	then   $|a_n|\le \lambda/(n-1)$ for $n\ge 2$. Equality holds for $\mathbf{N}_{F_{\lambda/(n-1),n-1}}$ defined by \eqref{EQ_N}.
\end{thm}
\begin{proof}
	Let $f\in \mathcal{N}(\lambda)$ be  given by \eqref{eq_expand_f}. Then the Corollary \ref{cor_N_lambda} implies that there exists $n\ge 1$ and  $h\in \mathcal{ST}_{ss}(\lambda/n)$ such that
	\[
	f(z)=\mathbf{N}_{g}(z)\quad \left(z\in \mathbb{D}\right).
	\]
	From \eqref{eq_N_f}, we obtain
	\[
	\sum_{n=2}^{\infty}a_nz^n=2\sum_{n=2}^{\infty}(n-1)\gamma_{n-1}(g)z^{n},
	\]
and applying Corollary \ref{cor_log_coff_n} we get $\left|\gamma_{n-1}(g)\right|\le \lambda/(n-1)^2$. Thus, we conclude $|a_n|\le \lambda/(n-1)$	and  inequality is sharp for  $g=F_{\lambda/(n-1),n-1}$. On the other hand $f=\mathbf{N}_{F_{\lambda/(n-1),n-1}}$. Note that for function $\mathbf{N}_{\lambda}(z)=z(1+z)^\lambda$, we have
	\[
	a_n=\frac{\lambda(\lambda-1)\cdots (\lambda-n+2)}{(n-1)!}\quad \left(n\ge 2\right)
	\]
and so $|a_n|\le \lambda/(n-1)$ holds.	
\end{proof}

\section{Coefficients bounds in  $\mathcal{ST}_{ss}(\lambda)$}\label{sec_class}
In this section, we determine a sharp upper bound on coefficients  and the Hankel determinant $H_2(2)=a_2a_4-a_3^2$ in class  $\mathcal{ST}_{ss}(\lambda)$. Also,  we find the sharp order of growth for the coefficients of functions in $\mathcal{ST}_{ss}(\lambda)$.

We recall first the well known result on the relations between initial coefficients in a well known class of Schwarz functions.
\begin{lem}\cite{LiberaZlot83}\label{lem_4}
Let $\omega$ be the Schwarz function with the power series $\omega(z)=\sum_{n=1}^{\infty}w_nz^n$. Then
 \begin{eqnarray*}
 w_2&=&\xi \myp{1-w_1^2},\\
 w_3&=&\myp{1-w_1^2}\myp{1-|\xi|^2}\zeta-w_1\myp{1-w_1^2}\xi^2,
 \end{eqnarray*}
 for some complex number $\xi$, $\zeta$ with $\left|\xi\right|\leq 1$ and
 $\left|\zeta\right|\leq 1$.
\end{lem}
\begin{thm}Let $f\in \mathcal{ST}_{ss}(\lambda)$ given by \eqref{eq_expand_f}. Then
\begin{equation*}
    \left|a_2a_4-a_3^2\right|\leq \dfrac{\lambda^2}{4}.
\end{equation*}
    The inequality is sharp.
\end{thm}

\begin{proof}
Let the function $f$ of the form \eqref{eq_expand_f} be in the class $\mathcal{ST}_{ss}(\lambda)$. Then there exists a
Schwarz function $\omega(z)=\sum_{n=1}^{\infty}w_nz^n$,  such that
\begin{equation}\label{eq:th12}
    \dfrac{zf'(z)}{f(z)}=\myp{1+\omega(z)}^{\lambda}.
\end{equation}
Equating coefficients  of both sides of \eqref{eq:th12} we obtain
\begin{eqnarray}
    a_2 &=&\lambda w_1\notag,\\[0.5em]
    a_3 &=&
    \dfrac{\lambda}{2}\myp{w_2+\dfrac{3\lambda-1}{2}w_1^2},\label{coffiecient}\\[0.5em]
    a_4 &=&
    \dfrac{\lambda}{3}\myp{w_3+\dfrac{5\lambda-2}{2}w_1w_2+\dfrac{17\lambda^2-15\lambda+4}{12}w_1^3}\notag.
\end{eqnarray}
Then we have
\begin{equation*}
    \left|a_2a_4-a_3^2\right|={}\dfrac{\lambda^2}{12}\left|w_1w_3-3w_2^2+\left\{\lambda-1\right\}w_1^2w_2+\dfrac{7-13\lambda^2-6\lambda}{12}w_1^4\right|.
\end{equation*}
Now, we use of  Lemma \ref{lem_4}, and write the expression $w_2$ and $w_3$ in terms of $w_1$. Since the Hankel determinant $H_2(2)$ is invariant under the rotation, thus without loss of generality we assume $x=w_1$ with $0\leq x\leq 1$. Next, applying triangle inequality, we obtain
\begin{multline*}
\left|a_2a_4-a_3^2\right|\leq\dfrac{\lambda^2}{12}\Biggl\{
\left|\dfrac{13\lambda^2+6\lambda-7}{12}\right|x^4+\left|\lambda-1\right|x^2\myp{1-x^2}|\xi|\\+3\myp{1-x^2}^2|\xi|^2
+4x\myp{1-x^2}\myp{1-|\xi|^2}+4x^2\myp{1-x^2}|\xi|^2\Biggr\}=:
g\myp{|\xi|}.
\end{multline*}
It is easy to check that the function $g(|\xi|)$ is increasing  on the interval  $[0,1]$. Thus $g(|\xi|)$
attains its  maximum at $|\xi|=1$, i.e., $g(|\xi|)\leq g(1)$. Consequently
\begin{align*}
\left|a_2a_4-a_3^2\right|&\leq
\dfrac{\lambda^2}{12}\left\{3-\myp{\lambda+1}x^2+\myp{\lambda-2+\left|\dfrac{-13\lambda^2-6\lambda+7}{12}\right|}x^4\right\},
\intertext{so that}
\left|a_2a_4-a_3^2\right|&\leq\left\{
\begin{array}{ll}
\dfrac{\lambda^2}{12}\left\{3-\myp{\lambda+1}x^2-\dfrac{13\lambda^2-6\lambda+17}{12}x^4\right\}\:\:
 &\textit{for}\quad 0<\lambda\leq \dfrac{7}{13},\\[1em]
\dfrac{\lambda^2}{12}\left\{3-\myp{\lambda+1}x^2-\dfrac{31-18\lambda -13\lambda^2}{12}x^4\right\}
 \:\: &\textit{for}\quad \dfrac{7}{13}\leq \lambda <1.
\end{array}\right.
\end{align*}
Since the expressions at $x^2$ and $x^4$ are negative in both cases, we obtain
$$\left|a_2a_4-a_3^2\right|\le \dfrac{\lambda^2}{4},$$
and the assertion follows. The function $F_{\lambda,2}$ or one of its rotations given in \eqref{F_n}, shows that the bound
$\lambda^2/4$ is sharp.
\end{proof}
\begin{thm}
Let the function $f$ of the form \eqref{eq_expand_f} belongs to the class $\mathcal{ST}_{ss}(\lambda)$. Then $\left|a_n\right|={\rm O}(1/n)$ for $n=1, 2, 3, \dots$.
\end{thm}
\begin{proof}
The functions $(1+z)^\lambda$ belong to the space of functions analytic and bounded in $\mathbb{D}$, denoted by $H^{\infty}$ \cite{MES}. Taking into account the results in \cite{MMin} we conclude the theorem.
\end{proof}
\begin{thm}\label{thm_FS_ST}
Let $f\in \mathcal{ST}_{ss}(\lambda)$ given by \eqref{eq_expand_f}. Then for  real number $\delta$, we have sharp inequalities
\begin{equation}\label{FSe}\left|a_3-\delta a_2^2\right| \leq \left\{
\begin{array}{ll}
    -\lambda^2\myp{\delta+\dfrac{1-3\lambda}{4\lambda}}\quad
    &\textit{for}\quad \delta <
    \dfrac{3\myp{\lambda-1}}{4\lambda},\\[0.5em]
    \dfrac{\lambda}{2}\quad &\textit{for}\quad
    \dfrac{3\myp{\lambda-1}}{4\lambda}\leq \delta \leq
    \dfrac{1+3\lambda}{4\lambda}, \\[0.5em]
    \lambda^2\myp{\delta+\dfrac{1-3\lambda}{4\lambda}}   \quad
    &\textit{for}\quad \delta > \dfrac{1+3\lambda}{4\lambda}.
\end{array}\right.
\end{equation}
\end{thm}
\begin{proof}
From \eqref{coffiecient} we have
\begin{equation*}
    \left|a_3-\delta
    a_2^2\right|=\dfrac{\lambda}{2}\left|w_2-\dfrac{4\delta \lambda
    -3\lambda+1}{2}w_1^2\right|.
\end{equation*}
Application of the estimates for the Fekete-Szeg\"{o} functional in the class of Schwarz functions  with $t=\myp{4\delta \lambda -3\lambda+1}/2$ \cite{ARS} establishes the  inequalities. When $\myp{3\lambda-3}/\myp{4\lambda}\leq \delta \leq
\myp{1+3\lambda}/\myp{4\lambda}$, equality holds 
 for $f$ is equal to $F_{\lambda,2}(z)$ given by
\eqref{F_n} or one of its rotation. If $ \delta >
\myp{1+3\lambda}/\myp{4\lambda}$ or
$\delta<\myp{3\lambda-3}/\myp{4\lambda}$ equality holds
 for $f$ is equal to  $F_{\lambda}(z)$ given by
\eqref{F_1} or one of its rotation. Let $F_x$ and $G_x$ be given by
\begin{equation}\label{FG}
\dfrac{zF'_x(z)}{F_x(z)}=\mathfrak{q}_{\lambda}\myp{z\dfrac{z+x}{1+xz}},\quad
\dfrac{zG'_x(z)}{G_x(z)}=\mathfrak{q}_{\lambda}\myp{-z\dfrac{z+x}{1+xz}},
\end{equation}
where  $x\in[0,1]$.  If $\delta=\myp{3\lambda-3}/\myp{4\lambda}$,  equality  in \eqref{FSe}
holds if and only if $f$ is equal to $F_x$ or one of its rotation,  and if $ \delta = \myp{1+3\lambda}/\myp{4\lambda}$, equality  in \eqref{FSe} holds if and only if $f$ is equal to $G_x$ or one of its rotation.
\end{proof}
The straightforward coefficient estimates follows from \eqref{coffiecient} and Theorem \ref{thm_FS_ST}, below.
\begin{cor}
If $f\in \mathcal{ST}_{ss}(\lambda)$ is of the form \eqref{eq_expand_f}, then
\[
|a_2|\le \lambda, \quad |a_3|\le \dfrac{\lambda}{2}.
\]
The first inequality is sharp for function $f(z)=\overline{\mu}F_{\lambda}(\mu z)$  and the second  for $f(z)=\overline{\mu}F_{\lambda,2}(\mu z)$, where $\mu$ is a unimodular complex number.
\end{cor}
\begin{cor}
Let the function $f\in \mathcal{ST}_{ss}(\lambda)$ be given by \eqref{eq_expand_f}. Then
\[
\left|a_3-a_2^2\right|\le \frac{\lambda}{2}.
\]
 The inequality is sharp for function $f(z)=\overline{\mu}F_{\lambda,2}(\mu z)$, where $\mu$ is a unimodular complex number.
\end{cor}
Let $f\in \mathcal{S}$ given by \eqref{eq_expand_f}. Then $F(z)=z/f(z)$ is a non-vanishing analytic function in $\mathbb{D}$ and
\begin{equation}\label{eq41}
    F(z)=
    \frac{z}{f(z)}=1+\sum_{n=1}^{\infty}c_nz^n=1-a_2z+(a_2^2-a_3)z^3+\cdots.
\end{equation}
\begin{thm}
    Let $f\in \mathcal{ST}_{ss}(\lambda)$  and $F(z)=z/f(z)$ given by \eqref{eq_expand_f} and \eqref{eq41},
    respectively. Then for  real number $\delta$, we have sharp inequalities for function $F(z)=z/f(z)$
\begin{equation*}\left|c_2-\updelta c_1^2\right| \leq \left\{
\begin{array}{ll}
    -\uplambda^2\myp{\updelta-\dfrac{\uplambda+1}{4\uplambda}}\quad
    &\quad\textit{for}\quad \updelta <
    \dfrac{\uplambda-1}{4\uplambda},\\[0.5em]
    \dfrac{\uplambda}{2}\quad &\quad\textit{for}\quad
    \dfrac{\uplambda-1}{4\uplambda}\leq \updelta \leq
    \dfrac{\uplambda+3}{4\uplambda}, \\[0.5em]
   \uplambda^2\myp{\updelta-\dfrac{\uplambda+1}{4\uplambda}}    \quad
    &\quad\textit{for}\quad \updelta > \dfrac{\uplambda+3}{4\uplambda}.
\end{array}\right.
\end{equation*}
\end{thm}
\begin{proof}
Let $f\in \mathcal{ST}_{ss}(\lambda)$ given by \eqref{eq_expand_f}.
From \eqref{eq41}, we have
\begin{equation*}
    \left|c_2-\updelta
    c_1^2\right|=\left|a_3-(1-\delta)a_1^2\right|.
\end{equation*}
Application of  Theorem \ref{thm_FS_ST} with $1-\delta$ gives the  inequalities. The inequalities are sharp for the functions
\begin{equation*}F(z)= \left\{
\begin{array}{ll}
        \dfrac{z}{\overline{\upmu} F_{\uplambda,2}(\upmu z)}\quad
        &\textit{for}\quad
        \dfrac{\uplambda-1}{4\uplambda}< \updelta <
        \dfrac{\uplambda+3}{4\uplambda},\\[0.5em]
       \dfrac{z}{\overline{\upmu} F_{\uplambda}(\upmu z)}\quad
       &\textit{for}\quad
       \updelta
       \in\myp{-\infty,
       \dfrac{\uplambda-1}{4\uplambda}}\cup\myp{\dfrac{\uplambda+3}{4\uplambda},\infty},\\[0.5em]
   \dfrac{z}{\overline{\upmu}G_x(\upmu z)}    \quad
  &\textit{for}\quad\updelta=\dfrac{\uplambda-1}{4\uplambda},\\[0.5em]
    \dfrac{z}{\overline{\upmu}F_x(\upmu z)}    \quad
    &\textit{for}\quad\updelta =
   \dfrac{\uplambda+3}{4\uplambda}.
\end{array}\right.
\end{equation*}
where $\upmu$ is an unimodular constant and  functions $F_x$ and $G_x\ \myp{0\le x\le 1}$ given by \eqref{FG}.
\end{proof}
\medskip

\noindent\textbf{Acknowledgments}
\noindent This work was partially supported by the Center for Innovation and Transfer of Natural Sciences and Engineering
Knowledge, Faculty of Mathematics and Natural Sciences, University of Rzeszow.\bigskip

\noindent\textbf{Authors' Contributions}

\noindent Each of the authors contributed to each part of this study equally, all authors read and approved the final manuscript.

\noindent\textbf{Competing Interests}

\noindent The authors declare that they have no competing interests.\bigskip


\end{document}